\documentclass[a4paper,11pt,reqno]{amsart}
\usepackage[a4paper, left=0.8in, right=0.7in, top=0.7in, bottom=0.7in]{geometry}
\usepackage{amsfonts}
\usepackage{dsfont}
\usepackage{amsmath}
\usepackage{array}
\usepackage{amssymb}
\usepackage{mathrsfs}
\usepackage{array}
\usepackage{comment}
\usepackage{csquotes}
\usepackage{enumitem}
\usepackage{mdframed}
\usepackage{xurl}
\usepackage{todonotes}
\usepackage{stackrel}
\usepackage{cite}

\newmdtheoremenv{lem}{Lemma}
\newmdtheoremenv{cor}{Corollary}
\newmdtheoremenv{rem}{Remark}
\usepackage{color}
\usepackage{multirow}
\allowdisplaybreaks

\usepackage[dvipsnames]{xcolor}
\usepackage[colorlinks = true,
            linkcolor = blue,
            urlcolor  = gray,
            citecolor = blue,
            anchorcolor = blue]{hyperref}
\renewcommand\eqref[1]{(\ref{#1})} %Need with hyperref

\newcommand{\M}[1]{%
    \par % Force a new line before
    \vspace{0.5cm} % Space before
    \todo[inline,color=blue!10!white,bordercolor=blue!50!black]{\textbf{Mukhtar:}~#1}
    \par % Force a new line after
    \vspace{0.5cm} % Space after
}

\newcommand{\A}[1]{%
    \par % Force a new line before
    \vspace{0.5cm} % Space before
    \todo[inline,color=red!10!white,bordercolor=red!50!black]{\textbf{Amin:}~#1}
    \par % Force a new line after
    \vspace{0.5cm} % Space after
}

\numberwithin{equation}{section}
\theoremstyle{plain}
\newtheorem*{assumptionA}{Assumption (A)}

\newtheorem{theorem}{Theorem}[section]

\newtheorem{corollary}[theorem]{Corollary}
\newtheorem{lemma}[theorem]{Lemma}
\newtheorem{remark}[theorem]{Remark}
\theoremstyle{definition}
\newtheorem{definition}[theorem]{Definition}

\usepackage{xurl} % allows breaks at - and _
\usepackage{ragged2e}

\begin{document}

   \title[]
   {Sharp homogeneous Gagliardo--Nirenberg inequalities and normalized
solutions for generalized stationary MMT equations}

\author[A. Esfahani]{Amin Esfahani}
\address{
  Amin Esfahani:
  \endgraf
  School of Computing and Artificial Intelligence (SCAI), 
  \endgraf
Nazarbayev University, Astana, Kazakhstan
  \endgraf
  {\it E-mail address:} {\rm  amin.esfahani@nu.edu.kz}
  }

\author[M. Karazym]{Mukhtar Karazym}
\address{
  Mukhtar Karazym:
  \endgraf
  Private Institution Nazarbayev University Research Administration, Astana, Kazakhstan
  \endgraf
\&
  \endgraf
School of  Artificial Intelligence and Data Science
\endgraf
Astana IT University, Astana, Kazakhstan
\endgraf
  {\it E-mail address:} {\rm  mukhtar.karazym@nu.edu.kz}
  }


\begin{abstract} 
We prove the existence of optimizers for a class of homogeneous
Gagliardo--Nirenberg inequalities of the form
$$
\|D^{s}\phi\|_{L^q}
\leq
C
\|D^{s_1}\phi\|_{L^p}^{1-\theta}
\|D^{s_2}\phi\|_{L^2}^{\theta},
\quad s_1<s_2,
$$
under the scaling relation
$$
\frac{1}{q}-\frac{s}{d}
=
(1-\theta)\left(\frac{1}{p}-\frac{s_1}{d}\right)
+
\theta\left(\frac{1}{2}-\frac{s_2}{d}\right),
$$
and the strict interpolation condition
$$
s<(1-\theta)s_1+\theta s_2.
$$
When $s>s_1$, the problem
reduces, after shifting the derivative orders, to the result
of Bellazzini, Frank and Visciglia, whereas the case $s\leq s_1$ follows from the present paper. As an application,
we study normalized solutions of the associated Euler--Lagrange
equations under the constraint
$$
\|D^{s_1}\psi\|_{L^p}^p=\lambda>0.
$$
In particular, the case $p=2$ includes the stationary equation arising from the Majda–McLaughlin–Tabak (MMT) model.

\end{abstract}
\maketitle
\setcounter{tocdepth}{3}
\setcounter{secnumdepth}{3}
%\tableofcontents

\section{Introduction}
 
The study of sharp constants and optimizers is not merely a question of
functional inequalities. Sharp Gagliardo--Nirenberg inequalities play an important role in the analysis
of global well-posedness, blow-up, and stability for nonlinear evolution
equations; see, for example, \cite{Kenig2006, Martel2002, weinstein1983nonlinear}. They also serve as a fundamental tool in the variational construction of
normalized solutions. From a functional-analytic perspective, the validity of Gagliardo--Nirenberg inequalities has been studied in various settings; see, for example, \cite{Brezis2018,Brezis2019} for the inhomogeneous
theory and \cite{Hajaiej2011} for homogeneous fractional inequalities.

To understand the main result proved here, we start with
\begin{equation}\label{eq:quotient-intro}
\frac{\|D^{s} u\|_{L^q}}
{\|D^{s_1}u\|_{L^p}^{1-\theta}
 \|D^{s_2}u\|_{L^2}^{\theta}},
\quad s_1<s_2,
\end{equation}
where, $D^{\sigma}=(-\Delta)^{\sigma / 2}$ denotes the homogeneous Fourier multiplier with symbol $|\xi|^\sigma$.
When $s_1<s<s_2$, setting $v=D^{s_1} u$ gives
$$
\frac{\|D^{s-s_1}v\|_{L^q}}
{\|v\|_{L^p}^{1-\theta}
 \|D^{s_2-s_1}v\|_{L^2}^{\theta}},
$$
whose maximizer is obtained in \cite{Bellazzini2014} under certain assumptions. For $s\leq s_1<s_2$, we define $v=D^{s} u$ so that existence of maximizers for
$$
\frac{\|v\|_{L^q}}
{\|D^{s_1-s}v\|_{L^p}^{1-\theta}
 \|D^{s_2-s}v\|_{L^2}^{\theta}}
$$
can be studied by our paper. Thus, these two cases
together yield the general unified result. The above substitutions are bijections between the corresponding
admissible classes and therefore, after maximizing, preserve both the sharp constants and their
attainment.  All powers of $D$ are understood modulo polynomials. Similar reductions can be found in \cite[Corollary 5.6]{Morosi2018}.

We now recall the findings in \cite{Bellazzini2014}. Under the assumptions
$$
1<p,q<\infty,\quad 0<s<s_2,\quad
\frac{s}{s_2}<\theta<1,
$$
the sharp constant
\begin{equation}\label{eq:GN-BFV}
C_{\mathrm{BFV}}
:=
\sup_{\substack{
\phi\in \dot{H}^{s_2}(\mathbb{R}^d)\cap L^p(\mathbb{R}^d)\\
\phi\neq 0
}}
\frac{
\|D^{s}\phi\|_{L^q}
}{
\|\phi\|_{L^p}^{1-\theta}
\|D^{s_2}\phi\|_{L^2}^{\theta}
}
\end{equation}
is attained. The endpoint case $s=0$, which is not covered by
\cite{Bellazzini2014}, has been studied in 
\cite{DelPino2002,Agueh2008,Zhang2021,FallWeth2026}.
These works establish attainment of the sharp constants under additional restrictions on the parameters.

To formulate the main objective of the paper, we first introduce the following
standing assumption on the parameters.

\begin{assumptionA}\label{ass:general}
The parameters $p, q, s, s_1, s_2\in \mathbb{R}$ satisfy
\begin{enumerate}
\item $1<p,q<\infty$;
\item $s_1<s_2$;
\item there exists $0<\theta<1$ such that
\begin{equation}\label{eq:scaling-relation}
\frac{1}{q}-\frac{s}{d}
=
(1-\theta)\left(\frac{1}{p}-\frac{s_1}{d}\right)
+
\theta\left(\frac{1}{2}-\frac{s_2}{d}\right);
\end{equation}
\item
\begin{equation}\label{eq:strict-interpolation}
s<(1-\theta)s_1+\theta s_2.
\end{equation}
\end{enumerate}
\end{assumptionA}

Under Assumption (A), the main objective of this paper is to prove that 
\begin{equation}\label{eq:GN-EK}
C_{*}:=
\sup_{\phi\neq0}
\frac{\|D^{s}\phi\|_{L^{q}}}
{\|D^{s_1}\phi\|_{L^p}^{1-\theta}
\|D^{s_2}\phi\|_{L^2}^{\theta}}
\end{equation}
is attained. The supremum in \eqref{eq:GN-EK} is taken over all nonzero admissible
tempered distributions $\phi$ such that
$$
D^{s_1}\phi\in L^p(\mathbb{R}^d),\quad D^{s_2}\phi\in L^2(\mathbb{R}^d).
$$

Besides the class of inequalities discussed above, particular cases of \eqref{eq:GN-EK} have appeared previously in different settings, derived by different approaches.
Using profile decomposition, Wang and the first author \cite{Wang2022}
proved the existence of maximizers for
$$
C_{\mathrm{WE}}
:=
\sup_{\substack{
\phi\in \dot{H}^{s_1}(\mathbb{R})
\cap \dot{H}^{s_1+1}(\mathbb{R})\\
\phi\neq0}}
\frac{
\|D^{s}\phi\|_{L^4}
}{
\|D^{s_1}\phi\|_{L^2}^{1-\theta}
\|D^{s_1+1}\phi\|_{L^2}^{\theta}
},
$$
where
$$
s,s_1\geq0,
\quad
-\frac{1}{4}<s-s_1<\frac{3}{4}.
$$

Recently, Hajaiej and Su \cite{hajaiej2026}, using a Weinstein functional
together with rearrangement and compactness arguments, proved that, for
$$
d\geq3,\quad 0<s_1<1,\quad
\frac{2d}{d-2s_1}<q<\frac{2d}{d-2},
$$
the sharp constant
\begin{equation}\label{eq:HS}
C_{\mathrm{HS}}
:=
\sup_{\substack{
\phi\in\dot H^{s_1}(\mathbb{R}^d)\cap\dot H^1(\mathbb{R}^d)\\
\phi\neq0}}
\frac{\|\phi\|_{L^{q}}}
{\|D^{s_1}\phi\|_{L^2}^{1-\theta}
 \|D\phi\|_{L^2}^{\theta}}
\end{equation}
is attained. Since $s=0<s_1$, we see that \eqref{eq:HS} lies in the regime covered by 
Theorem \ref{thm:GN-optimizer}. Moreover, we provide an alternative proof of the attainment for these parameters.

Another sharp inequality with an optimizer due to Bieganowski and Schino
\cite[Theorem 3.3]{Bieganowski2026}, namely,
\begin{equation}\label{eq:GN-by-BS2026}
C_{\mathrm{BS}}:= \sup_{\substack{\phi\in H^{s_2}(\mathbb{R}^d)\\
\phi\neq0}}\frac{\|\phi\|_{L^{q}}}{\|\phi\|_{L^2}^{1-\theta}\|D^{s_2}\phi\|_{L^2}^{\theta}
}
\end{equation}
is recovered from \eqref{eq:GN-EK}. Their proof is based on the minimization of a Weinstein-type functional, with
compactness obtained through Lions-type concentration--compactness argument
\cite{Lions1984part2}, whereas our approach supplies a different proof.

For comparison, we also mention a quantitative result of Hott
\cite{Hott2026}, who obtained explicit upper bounds for a broad
class of non-endpoint homogeneous fractional
Gagliardo--Nirenberg--Sobolev inequalities. In particular, his results provide explicit nonsharp upper bounds for the constant in \eqref{eq:GN-EK} whenever 
\begin{equation}\label{eq:eta}
\eta:=s_2-s_1+\frac{d}{p}-\frac{d}{2}\neq 0.
\end{equation}

As an application of the sharp Gagliardo--Nirenberg inequality
associated with \eqref{eq:GN-EK}, our next objective is to study critical points of the
energy functional
\begin{equation}\label{eq:energy}
E(\psi):=
\frac{1}{2}\int_{\mathbb{R}^d}|D^{s_2}\psi|^2 dx
-
\frac{1}{q}\int_{\mathbb{R}^d}|D^{s}\psi|^q dx
\end{equation}
under the prescribed constraint
\begin{equation}\label{eq:M}
M(\psi):=\int_{\mathbb{R}^d}|D^{s_1}\psi|^p dx=\lambda>0.
\end{equation}
According to the $M$-preserving scaling of the energy $E$ on the constraint above, we treat
separately the $M$-subcritical, $M$-critical, and $M$-supercritical
cases
$$
q\theta<2,\quad q\theta=2,\quad q\theta>2,
$$
respectively. The optimizers of \eqref{eq:GN-EK} play a central role in
the corresponding variational problems and, after a suitable
normalization, give rise to solutions of the following Euler--Lagrange equation
\begin{equation}\label{eq:general-Euler-Lagrange}
D^{2s_2}\psi
-
D^{s}\left(
|D^{s}\psi|^{q-2}D^{s}\psi
\right)
+
\omega D^{s_1}\left(
|D^{s_1}\psi|^{p-2}D^{s_1}\psi
\right)
=0
\quad \text{in } \mathbb{R}^d,
\end{equation}
where $\omega>0$ appears as the Lagrange multiplier. 
We may assume without loss of
generality that $s_2>0$ in \eqref{eq:general-Euler-Lagrange}. Indeed, fix $\sigma>0$ and let
$$
\widetilde\psi:=D^{s_2-\sigma}\psi.
$$
Applying $D^{\sigma-s_2}$ to \eqref{eq:general-Euler-Lagrange}, we obtain
\begin{equation}\label{eq:transformed-Euler-Lagrange}
D^{2\sigma}\widetilde\psi
-
D^{\widetilde{s}}\big(
|D^{\widetilde{s}}\widetilde\psi|^{q-2}
D^{\widetilde{s}}\widetilde\psi
\big)
+
\omega D^{\widetilde{s}_1}\big(
|D^{\widetilde{s}_1}\widetilde\psi|^{p-2}
D^{\widetilde{s}_1}\widetilde\psi
\big)
=0,
\end{equation}
where
$$
\widetilde{s}:=s-s_2+\sigma,\quad
\widetilde{s}_1:=s_1-s_2+\sigma,\quad
\widetilde{s}_2:=\sigma
$$
still satisfy Assumption (A).

If $\eta$ given by \eqref{eq:eta} is nonzero,
then
the parameter $\omega$ can be normalized to one. Indeed, whenever
$\psi_{\omega}$ satisfies \eqref{eq:general-Euler-Lagrange}, then
$$
\widetilde\psi(x)
:=
\omega^{\frac{qs-2s_2}{2pq\eta\delta}}
\psi_{\omega}\left(
\omega^{\frac{2-q}{2pq\eta\delta}}x
\right),
\quad
\delta
:=
\frac{\theta}{2}
+\frac{1-\theta}{p}
-\frac{1}{q},
$$
solves \eqref{eq:general-Euler-Lagrange} with $\omega=1$.

The connection to the MMT model is evident when $p=2$.
We now review the known results that are most closely related to \eqref{eq:general-Euler-Lagrange}. Setting
$$
\varphi=D^{s_1}\psi
$$
in \eqref{eq:general-Euler-Lagrange} and applying $D^{-s_1}$ gives
\begin{equation}\label{eq:MMT}
D^{\alpha}\varphi
-
D^{-\beta/4}\big(
|D^{-\beta/4}\varphi|^{2}D^{-\beta/4}\varphi
\big)
+\omega\varphi=0
\quad \text{in } \mathbb{R},
\end{equation}
where
$$
\alpha=2(s_2-s_1), \quad \beta=4(s_1-s).
$$
This is the stationary equation associated with the MMT model
introduced in \cite{Majda1997}. Thus,
$s>s_1$, $s=s_1$, and $s<s_1$ correspond respectively to
$\beta<0$, $\beta=0$, and $\beta>0$. In the present setting, the scaling relation \eqref{eq:scaling-relation} becomes
$$
\beta=1-2\theta\alpha,
$$
which in turn yields
\begin{equation}\label{eq:MMT-range}
1-2\alpha<\beta<1.
\end{equation}

Zakharov et al. \cite{Zakharov2001} investigated the role of wave
collapses and quasisolitons in wave turbulence for the cubic MMT model
\eqref{eq:MMT}. In our convention for the parameter $\beta$, their
soliton stability threshold is
$$
\beta=1-\alpha,
$$
arising from their soliton stability criterion, coincides with the
$M$-critical threshold $q\theta=2$ of the present variational problem. Accordingly, the soliton-stable and soliton-unstable sides correspond
to $q\theta<2$ and $q\theta>2$, respectively. The latter also corresponds
to the parameter regime for self-similar collapse identified by Cai et al. \cite{Cai2001}. This correspondence is consistent with the standard variational
classification for focusing dispersive equations: mass-subcritical waves
are typically stable, whereas mass-supercritical waves are unstable; see, e.g.
\cite{Berestycki1981, Weinstein1985, Weinstein1986, Dinh2019}.

More recently, the well-posedness of the initial value problem for the MMT model associated with the stationary equation \eqref{eq:MMT} was studied in \cite{panthee2026}. 
The corresponding periodic problem was further studied in
\cite{Panthee2026torus}.  We also mention that Hrabski and Pan \cite{hrabski2022breather} numerically investigated breather
solutions for a two-dimensional periodic MMT model and revealed a
connection between their phase-space dynamics and Kolmogorov--Arnold--Moser theory.
From a different perspective, Germain, La, and Zhang
\cite{Germain2025} developed a local well-posedness theory for the
kinetic equation associated with the MMT model. More recently, L\"u
\cite{Lu2026} studied the wave kinetic theory for the quasilinear MMT
equation itself and established long-time regularity and a kinetic
approximation for random initial data.

More generally, one may replace the cubic nonlinearity in
\eqref{eq:MMT} by a power-type nonlinearity and consider
\begin{equation}\label{eq:MMT-power}
D^{\alpha}\varphi - D^{-\beta/4}\big(|D^{-\beta/4}\varphi|^{q-2}D^{-\beta/4}\varphi\big) + \omega\varphi=0
\quad\text{in }\mathbb{R}.
\end{equation}
The first author and Muslu \cite{Esfahani2026} studied analytically
and numerically the case $\alpha=2$ of \eqref{eq:MMT-power}.
They established the existence of a ground state for
$$ 
-2-\frac{4}{q}<\beta<2-\frac{4}{q},
$$
whereas no nontrivial standing wave solution exists outside this range.

\subsection{Main results}\label{subsection}
Our first result concerns the best constant $C_{*}>0$ in 
\begin{equation}\label{eq:GN-sharp}
\|D^{s}\phi\|_{L^q}
\leq 
C_{*}
\|D^{s_1}\phi\|_{L^p}^{1-\theta}
\|D^{s_2}\phi\|_{L^2}^{\theta}
\end{equation}
for all admissible $\phi$.

\begin{theorem}\label{thm:GN}
Under Assumption (A), the sharp constant
\begin{equation}\label{eq:best-const}
C_{*}:=
\sup_{\phi\neq0}
\frac{\|D^{s}\phi\|_{L^q}}
{\|D^{s_1}\phi\|_{L^p}^{1-\theta}
\|D^{s_2}\phi\|_{L^2}^{\theta}}
\end{equation}
is attained. Here, the supremum is taken over all nonzero 
$\phi\in\mathcal{S}'(\mathbb{R}^d)$ such that
$$
D^{s_1}\phi\in L^p(\mathbb{R}^d),\quad D^{s_2}\phi\in L^2(\mathbb{R}^d).
$$
\end{theorem}

The remainder of this subsection is devoted to the variational consequences of Theorem \ref{thm:GN}, namely, the existence of normalized solutions for the associated Euler--Lagrange equation \eqref{eq:general-Euler-Lagrange}. Throughout the introduction, $\mathcal{X}$ denotes the homogeneous space defined in Section \ref{sec2}.  For convenience, we introduce $A,B:\mathcal{X}\to \mathbb{R}$ given by
$$
A(\psi):=\|D^{s_2}\psi\|_{L^2}^2,\quad B(\psi):=\|D^{s}\psi\|_{L^q}^q.
$$
Then \eqref{eq:energy} can be rewritten as
$$
E(\psi)=\frac{1}{2}A(\psi)-\frac{1}{q}B(\psi).
$$
For $\lambda>0$, let
$$
\mathcal{S}_{\lambda}:=\{\psi\in \mathcal{X}:M(\psi)=\lambda\}.
$$
The relevant scaling is the one that preserves the constraint
\eqref{eq:M}. 
Namely, given  $\rho>0$, we define 
\begin{equation}\label{eq:scaling}
(\rho\star\psi)(x):=\rho^{d/p-s_1}\psi(\rho x).
\end{equation}
Then
$$
M(\rho\star\psi)=M(\psi)
$$
and
\begin{equation}\label{eq:energy-scaling}
E(\rho\star\psi)=\frac{\rho^{2\eta}}{2} A(\psi)-\frac{\rho^{q\theta\eta}}{q} B(\psi),
\end{equation}
where $\eta$ is given by \eqref{eq:eta}.
\begin{remark}
Compared with the usual normalized NLS framework, the scaling \eqref{eq:scaling} allows a broader range of amplitude exponents. Indeed, for an $L^2$-constraint, the mass-preserving scaling has the fixed amplitude factor $\rho^{d / 2}$:
$$
\psi_{\rho}(x)=\rho^{d/2}\psi(\rho x).
$$
In our case, the derivative constraint $M(\psi)=\lambda$
is preserved by \eqref{eq:scaling},
whose amplitude exponent
\begin{equation}\label{eq:amplitude-exponent}
\frac{d}{p}-s_1
\end{equation}
may take any real value. 

Moreover, the exponent $\eta$ governing the
energy scaling \eqref{eq:energy-scaling} may be
positive, zero, or negative. This leads to genuinely different scaling geometries: the roles of
$\rho\to 0$ and $\rho\to\infty$ are reversed when $\eta$ changes sign,
whereas $\eta=0$ corresponds to the scale-invariant case. Such a phenomenon does not occur in the standard $L^2$-normalized NLS
and fractional NLS settings, where $s_1=0$, $p=2$, and $s_2>0$. In this case,
$$
\eta=s_2>0.
$$

Finally, we note that, since $s_1<s_2$, we have
$$
p \leq 2 \quad \Rightarrow \quad \eta>0.
$$
Thus, $\eta \leq 0$ can occur if and  $p>2$. 
\end{remark}
Since $\eta= 0$ makes $E$ invariant under \eqref{eq:scaling}, we first consider $\eta \neq 0$. Let
$$
\kappa:=d+p(s-s_1).
$$
The scaling relation \eqref{eq:scaling-relation} can be written as
\begin{equation}\label{eq:scaling-with-kappa}
\frac{1}{q}=\frac{\kappa}{pd}-\frac{\theta\eta}{d}>0,
\end{equation}
from which we have
$$
q\theta-2=\frac{q\kappa-2\bigl(d+p(s_2-s_1)\bigr)}{p\eta}.
$$

Suppose first that $\kappa>0$. Then we define
$$
q_0:=\frac{pd}{\kappa}
$$
and 
$$
q_S:=
\begin{cases}
\dfrac{2d}{d-2(s_2-s)},
&\quad \text{if } s_2-s<\dfrac{d}{2},\\
+\infty,
&\quad \text{if } s_2-s\ge\dfrac{d}{2}.
\end{cases}
$$
The condition $0<\theta<1$ is equivalent to $q$ lying
strictly between $q_0$ and $q_S$. We further introduce the $M$-critical exponent
$$
q_c:=\frac{2\bigl(d+p(s_2-s_1)\bigr)}{\kappa}.
$$
Then
$$
q\theta-2
=
\frac{\kappa}{p\eta}\bigl(q-q_c\bigr).
$$
Consequently, when $\eta>0$, the $M$-subcritical,
$M$-critical, and $M$-supercritical regimes correspond,
respectively, to
$$
q<q_c,\quad q=q_c,\quad q>q_c,
$$
whereas when $\eta<0$, these inequalities are reversed.

If $\kappa\leq0$, the admissibility condition \eqref{eq:scaling-with-kappa} necessarily implies $\eta<0$. Since
$$
q\kappa- 2\bigl(d+p(s_2-s_1)\bigr)
<0,
$$
it follows that there is only $M$-supercritical case.

Since we have three different variational regimes, we state the corresponding
results separately.

\begin{theorem}[$M$-subcritical case]\label{thm:subcritical}
Under Assumption (A), and
$$
\eta\neq 0, \quad q\theta<2,
$$
the infimum
$$
I_{\lambda}
=
\inf_{\psi\in\mathcal{S}_{\lambda}}E(\psi)
$$
is attained for every $\lambda>0$. Moreover,
\begin{equation}\label{eq:Ilambda-explicit-subcritical}
I_{\lambda}
=
-\frac{2-q\theta}{2q\theta}
\left(
\theta C_{*}^q
\lambda^{\frac{(1-\theta)q}{p}}
\right)^{
\frac{2}{2-q\theta}
}
<0.
\end{equation}
\end{theorem}

\begin{theorem}[$M$-critical case]\label{thm:critical}
Under Assumption (A),  and
$$
\eta\neq 0, \quad q=q_c,
$$
we have
$$
I_{\lambda}=
\inf_{\psi\in\mathcal{S}_{\lambda}}E(\psi)
=
\begin{cases}
0, & \quad \text{if }\,  0<\lambda\leq  \lambda_*,\\
-\infty, &\quad \text{if }\,  \lambda>\lambda_*.
\end{cases}
$$
Here 
$$
\lambda_*
:=
\left(
\frac{q}{2C_{*}^q}
\right)^{
\frac{p}{(1-\theta)q}
}
$$
with $C_{*}$ given by \eqref{eq:best-const}. 
Moreover, the infimum is attained if and only if
$$
\lambda=\lambda_*.
$$
\end{theorem}

In the supercritical case, direct minimization of $E$ on
$\mathcal{S}_{\lambda}$ is no longer possible, since $E$ is unbounded
from below along the $M$-preserving scaling \eqref{eq:scaling}.
To overcome this, we introduce the Pohozaev functional
$$
P(\psi):=A(\psi)-\theta B(\psi),
$$
and the Pohozaev manifold
$$
\mathcal{P}_{\lambda}
:=
\{\psi\in\mathcal{S}_{\lambda}:P(\psi)=0\}.
$$
As will be shown in Lemma \ref{lem:Pohozaev projection}, each scaling fiber
$$
\{\rho\star\psi:\rho>0\}
$$
intersects $\mathcal{P}_{\lambda}$ at exactly one point, and this point is the
unique maximizer of $E$ along the fiber. Hence, it is natural to consider the following
minimization problem
\begin{equation}\label{eq:min-on-Pohozaev}
m_{\lambda}
:=
\inf_{\psi\in\mathcal{P}_{\lambda}}E(\psi).
\end{equation}

The next theorem gives the variational replacement for direct minimization in the supercritical case: the minimization problem on $\mathcal{P}_{\lambda}$ is attained, and its minimizers are critical points of $E$ on $\mathcal{S}_{\lambda}$.

\begin{theorem}[$M$-supercritical case]\label{thm:supercritical}
Under Assumption (A), and
$$
\eta\neq 0, \quad \theta q>2,
$$
the minimization
problem \eqref{eq:min-on-Pohozaev} is attained for every $\lambda>0$. Moreover, every minimizer $\psi_{\lambda}\in\mathcal{P}_{\lambda}$ is a critical point of $E$ on $\mathcal{S}_{\lambda}$. In addition,
$$
E(\psi_{\lambda})\to+\infty
\quad\text{as }\lambda\to 0^+,
$$
and
$$
E(\psi_{\lambda})\to 0
\quad\text{as }\lambda\to+\infty.
$$
\end{theorem}

When $\eta=0$, the constraint $M$ and both terms of the energy $E$ are invariant under the same scaling \eqref{eq:scaling}, so the usual scaling argument does not provide compactness. Nevertheless, in the case $s \leq s_1$, the Hardy-Littlewood-Sobolev inequality (or the identity when $s=$ $s_1$) controls the nonlinear term $B$ solely by the constraint $M$, yielding coercivity of $E$ on $\mathcal{S}_{\lambda}$.  Moreover, the critical Sobolev inequality provides a uniform
lower bound for the quadratic term $A$, which leads to strict subadditivity.
Together with compactness up to translations and dilations at the critical Sobolev exponent \cite{Palatucci2014}, these properties allow us to obtain a nonzero weak
limit and exclude a nontrivial splitting of $M$. Thus, we obtain the following attainment result.

\begin{theorem}\label{thm:eta-zero}
Under Assumption (A), and
$$
\eta =0, \quad s\leq s_1,
$$
the infimum
$$
I_{\lambda}
:=
\inf_{\psi\in\mathcal{S}_{\lambda}}E(\psi)
$$
is attained for every $\lambda>0$.
\end{theorem}

\begin{remark}
All three $M$-regimes may occur under the assumptions of Theorem \ref{thm:eta-zero}. Indeed, when $\eta=0$, the scaling relation \eqref{eq:scaling-relation} does not determine $\theta$, and the same variational problem may be represented with $q\theta<2$, $q\theta=2$, or $q\theta>2$. 
\end{remark}

Lastly, we collect the Euler--Lagrange consequences of the preceding existence
results in a single statement.

\begin{corollary}\label{cor:Euler-Lagrange}
Let $\psi_\lambda$ be a minimizer given by one of
Theorems \ref{thm:subcritical}, \ref{thm:supercritical},
or \ref{thm:eta-zero}, or by
Theorem \ref{thm:critical} when $\lambda=\lambda_*$.  Then there exists a Lagrange multiplier $\omega_\lambda\in\mathbb{R}$
such that
$$
D^{2s_2}\psi_\lambda
-
D^s\left(
|D^s\psi_\lambda|^{q-2}D^s\psi_\lambda
\right)
+
\omega_\lambda D^{s_1}\left(
|D^{s_1}\psi_\lambda|^{p-2}D^{s_1}\psi_\lambda
\right)
=0
$$
in $\mathcal{X}'$, with
$$
M(\psi_\lambda)=\lambda.
$$

If $\eta\neq0$, then
$$
\omega_\lambda
=
\frac{1-\theta}{\theta\lambda}A(\psi_\lambda)>0.
$$

If $\eta=0$, then
$$
\omega_\lambda
=
\frac{B(\psi_\lambda)-A(\psi_\lambda)}{\lambda}<0 
$$
for all  
\begin{equation}\label{eq:range of lambda}
0<\lambda< \left(
\frac{C_{\mathrm{S}}^{-2}}{C_{\mathrm{H}}^q}
\right)^{\frac{p}{q-2}},
\end{equation}
where $C_{\mathrm{S}}$ and $C_{\mathrm{H}}$ are the sharp constants in the Sobolev and
Hardy--Littlewood--Sobolev inequalities
$$
\|D^{s_1}\psi\|_{L^p} \leq C_{\mathrm{S}}\|D^{s_2}\psi\|_{L^2},
\quad \|D^s\psi\|_{L^q} \leq C_{\mathrm{H}}\|D^{s_1}\psi\|_{L^p},
$$
respectively; see 
\cite{Cotsiolis2004, Lieb1983}. 
\end{corollary}

\section{Preliminaries}\label{sec2}

We denote by $\mathcal{S}(\mathbb{R}^d)$ the Schwartz space and by
$\mathcal{S}'(\mathbb{R}^d)$ its dual, the space of tempered distributions.
We use $\langle\cdot,\cdot\rangle$ for duality pairings, with the underlying
spaces understood from the context.

Throughout the paper, $\mathcal{F}$ denotes the Fourier transform
$$
\mathcal{F}\phi(\xi)
:=
(2\pi)^{-d/2}
\int_{\mathbb{R}^d}e^{-ix\cdot\xi}\phi(x) dx,
$$
with inverse
$$
\mathcal{F}^{-1}\phi(x)
:=
(2\pi)^{-d/2}
\int_{\mathbb{R}^d}e^{ix\cdot\xi}\phi(\xi) d\xi.
$$
We also write $\widehat{\phi}:=\mathcal{F}\phi$. The Fourier transform is
first defined on $\mathcal{S}(\mathbb{R}^d)$ and extended to
$\mathcal{S}'(\mathbb{R}^d)$ by duality.

Since homogeneous Sobolev spaces are naturally defined modulo polynomials,
we work in the quotient
\begin{equation}\label{eq:quotient}
\mathcal{S}'(\mathbb{R}^d)/\mathcal{P}(\mathbb{R}^d),
\end{equation}
where $\mathcal{P}(\mathbb{R}^d)$ denotes the space of polynomials on
$\mathbb{R}^d$. For $T\in\mathcal{S}'(\mathbb{R}^d)$, we denote by $[T]$
its equivalence class modulo polynomials.

Recall that \eqref{eq:quotient} is canonically identified with the dual
$\mathcal{S}'_\infty(\mathbb{R}^d)$ of the Lizorkin space
$$
\mathcal{S}_\infty(\mathbb{R}^d)
:=
\left\{
\varphi\in\mathcal{S}(\mathbb{R}^d):
\int_{\mathbb{R}^d}x^\alpha\varphi(x) dx=0
\text{ for every multi-index }\alpha
\right\};
$$
see, for example, \cite{Samko2002,Sawano2017}.

Given $s\in\mathbb{R}$ and $\varphi\in\mathcal{S}_\infty(\mathbb{R}^d)$, we define
$$
D^{s}\varphi
:=
\mathcal{F}^{-1}\bigl(|\xi|^s\widehat{\varphi}\bigr).
$$
The operator $D^{s}$ is an isomorphism of
$\mathcal{S}_\infty(\mathbb{R}^d)$ onto itself, with inverse $D^{-s}$;
see, for example, \cite{Monguzzi2020}.
It therefore extends by duality to
$\mathcal{S}'_\infty(\mathbb{R}^d)$. Through the natural isomorphism
$$
\mathcal{S}'_\infty(\mathbb{R}^d)
\simeq
\mathcal{S}'(\mathbb{R}^d)/\mathcal{P}(\mathbb{R}^d),
$$
we regard $D^{s}$ as an operator on
$\mathcal{S}'(\mathbb{R}^d)/\mathcal{P}(\mathbb{R}^d)$.

Given $1<\ell<\infty$, we first describe how $L^\ell(\mathbb{R}^d)$ is viewed inside the quotient \eqref{eq:quotient},
then define the homogeneous Sobolev spaces $\dot{L}_s^\ell(\mathbb{R}^d)$.

\begin{definition}
Let $1<\ell<\infty$. We define $\dot{L}^\ell(\mathbb{R}^d)$ as the set
of all equivalence classes
$$
[\phi]\in \mathcal{S}'(\mathbb{R}^d)/\mathcal{P}(\mathbb{R}^d)
$$
which admit an $L^\ell(\mathbb{R}^d)$-representative. Since the only polynomial
belonging to $L^\ell(\mathbb{R}^d)$ is zero, this representative is unique. Thus,
if $[\phi]=[\widetilde\phi]$ with $\widetilde\phi\in L^\ell(\mathbb{R}^d)$, we set
$$
\|[\phi]\|_{\dot{L}^\ell}:=\|\widetilde\phi\|_{L^\ell}.
$$
\end{definition}

\begin{definition}
Let $s\in\mathbb{R}$ and $1<p<\infty$. The homogeneous Sobolev space
$\dot{L}_s^\ell(\mathbb{R}^d)$ is the space of all
$$
\phi\in \mathcal{S}'(\mathbb{R}^d)/\mathcal{P}(\mathbb{R}^d)
$$
such that the distribution
$$
\mathcal{F}^{-1}\left(|\xi|^s\mathcal{F} \phi\right)
$$
belongs to $\dot{L}^\ell(\mathbb{R}^d)$. For $\phi\in \dot{L}_s^\ell(\mathbb{R}^d)$, we set
$$
\|\phi\|_{\dot{L}_s^\ell}
:=
\left\|
\mathcal{F}^{-1}\left(|\xi|^s\mathcal{F} \phi\right)
\right\|_{\dot{L}^\ell}.
$$
\end{definition}

Since we may assume that $s_2>0$ for \eqref{eq:general-Euler-Lagrange}, we can use the realization
$E^{s_2,2}(\mathbb{R}^d)$ of the homogeneous Sobolev space
$\dot L_{s_2}^2(\mathbb{R}^d)$ introduced in
\cite[Definition 1.14]{Monguzzi2020}. By
\cite[Theorem 3]{Monguzzi2020}, every equivalence class in
$\dot L_{s_2}^2(\mathbb{R}^d)$ has a unique representative in
$E^{s_2,2}(\mathbb{R}^d)$, which removes polynomial ambiguity. Hence, identifying each $[\psi]\in\mathcal{X}$
with this representative, we may write
$$
\mathcal{X}:
=
\left\{
\psi\in E^{s_2,2}(\mathbb{R}^d):
D^{s_1}\psi\in L^p(\mathbb{R}^d)
\right\},
$$
with
$$
\|\psi\|_{\mathcal{X}}
:=\|D^{s_1}\psi\|_{L^p} + \|D^{s_2}\psi\|_{L^2}.
$$

\section{Proof of Theorem \ref{thm:GN}}

We first prove a compactness lemma which, together with the Brezis--Lieb
lemma, yields the splitting of the nonlinear term in the energy functional $E$ in the proof of
Theorem \ref{thm:GN-optimizer}. Related contributions can be found in
\cite[Theorem 6.13]{Leoni2023fractional} and
\cite[Theorem 3.1]{bonder2019}.

\begin{lemma}\label{lem:local-compactness-Lp2}
Let $1<p_1,p_2<\infty$, $\sigma>0$, and let $(v_n)$ satisfy
$$
\sup_n\|v_n\|_{L^{p_1}}<\infty,
\quad
\sup_n\|D^{\sigma} v_n\|_{L^{p_2}}<\infty.
$$
Then $(v_n)$ is relatively compact in
$L^{p_2}_{\mathrm{loc}}(\mathbb{R}^d)$. In particular, after passing to a
subsequence, there exists $v\in L^{p_2}_{\mathrm{loc}}(\mathbb{R}^d)$ such that
$$
v_n\to v
\quad\text{strongly in }L^{p_2}_{\mathrm{loc}}(\mathbb{R}^d).
$$
\end{lemma}

\begin{proof}
Let $K\Subset\mathbb{R}^d$. We prove that $(v_n)$ is relatively compact
in $L^{p_2}(K)$. Let $\chi\in C_c^\infty(\mathbb{R}^d)$ be a smooth cutoff function
such that
$$
\chi(\xi)=
\begin{cases}
1, & |\xi|\leq 1,\\
0, & |\xi|\geq  2.
\end{cases}
$$
For $\varepsilon>0$, we define
$$
S_\varepsilon f
:=
\mathcal{F}^{-1}
\big(
\chi(\varepsilon\xi)\mathcal{F} f(\xi)
\big).
$$
The idea is to split $v_n$ into low and high frequencies:
$$
v_n=S_{\varepsilon} v_n+(I-S_{\varepsilon}) v_n.
$$
The high-frequency part $(I-S_\varepsilon)v_n$ is uniformly small in
$L^{p_2}(K)$ as $\varepsilon\to 0$, while for fixed $\varepsilon>0$, the
low-frequency part $S_\varepsilon v_n$ is smooth, uniformly bounded, and
equicontinuous on $K$. Combining these two facts gives total boundedness of
$(v_n)$ in $L^{p_2}(K)$.

We first estimate the high-frequency part. Let
$$
m(\xi):=(1-\chi(\xi))|\xi|^{-\sigma}.
$$
Since  $m$ 
satisfies the Mihlin estimates, the multipliers with symbols
$m(\varepsilon\xi)$ are uniformly bounded on $L^{p_2}(\mathbb{R}^d)$.
Moreover,
$$
(I-S_\varepsilon)v_n
=
\varepsilon^\sigma
\mathcal{F}^{-1}\big(
m(\varepsilon\xi)\mathcal{F}D^{\sigma} v_n(\xi)
\big).
$$
Therefore,
$$
\|(I-S_\varepsilon)v_n\|_{L^{p_2}}
\le
C\varepsilon^\sigma
\|D^{\sigma} v_n\|_{L^{p_2}},
$$
where $C$ is independent of $\varepsilon$ and $n$.

It remains to prove compactness of the low-frequency part for fixed
$\varepsilon>0$. Since
$$
S_\varepsilon v_n=K_\varepsilon*v_n
$$
is  convolution with the Schwartz kernel  $K_\varepsilon=\mathcal{F}^{-1} \chi(\varepsilon\xi)$, we have
$$
\partial^\alpha S_\varepsilon v_n
=
(\partial^\alpha K_\varepsilon)*v_n\quad \text{ for every multi-index }\alpha.
$$
By Young's inequality,
$$
\|\partial^\alpha S_\varepsilon v_n\|_{L^\infty}
\leq 4
\|\partial^\alpha K_\varepsilon\|_{L^{p_1'}}
\|v_n\|_{L^{p_1}}.
$$
Thus, for every fixed $\varepsilon>0$, the sequence $(S_\varepsilon v_n)$ is uniformly bounded and equicontinuous on the compact set $K$. By the Arzelà--Ascoli
theorem, $(S_\varepsilon v_n)$ is relatively compact in $C(K)$, hence also
in $L^{p_2}(K)$.

Now let $\delta>0$. We choose $\varepsilon>0$ sufficiently small such that
$$
\sup_n
\|(I-S_\varepsilon)v_n\|_{L^{p_2}(K)}
<
\delta.
$$
For this fixed $\varepsilon>0$, the sequence 
$(S_\varepsilon v_n)$ is relatively compact in $L^{p_2}(K)$, and hence totally bounded.
Therefore, for every $\delta>0$, there exist $M\in\mathbb{N}$ and $w_1,\ldots,w_M\in L^{p_2}(K)$ such that 
$$
(S_\varepsilon v_n) \subset \bigcup_{j=1}^M B_{\delta}(w_j),
$$ 
where $B_\delta(w_j)$ denotes the open ball in $L^{p_2}(K)$
centered at $w_j$ with radius $\delta$.
Equivalently, for every $n\in\mathbb{N}$, there exists $j\in\{1,\ldots,M\}$ such that 
$$
\|S_\varepsilon v_n-w_j\|_{L^{p_2}(K)}<\delta. 
$$
Then
$$
\|v_n-w_j\|_{L^{p_2}(K)} \leq 
\|(I-S_\varepsilon)v_n\|_{L^{p_2}(K)}
+\|S_\varepsilon v_n-w_j\|_{L^{p_2}(K)}<2\delta.
$$
Thus, $(v_n)$ is totally bounded in $L^{p_2}(K)$. Since $L^{p_2}(K)$ is
complete, $(v_n)$ is relatively compact in $L^{p_2}(K)$.

Finally, by a diagonal argument, up to a subsequence,
$$
v_n\to v
\quad\text{strongly in }L^{p_2}_{\mathrm{loc}}(\mathbb R^d).
$$
\end{proof}

\begin{theorem}\label{thm:GN-optimizer}
Let $1<p,q<\infty$, $\mu\geq0$, and $\sigma>0$.
Assume that there exists $\theta\in(0,1)$ such that
$$
\frac{1}{q}=(1-\theta)\left(\frac{1}{p}-\frac{\mu}{d}\right)+\theta\left(\frac{1}{2}-\frac{\mu+\sigma}{d}\right).
$$
Then the sharp constant
\begin{equation}\label{eq:C_1}
C_1:=\sup_{\phi\neq0}\frac{\|\phi\|_{L^q}}{\|D^{\mu}\phi\|_{L^p}^{1-\theta}\|D^{\mu+\sigma}\phi\|_{L^2}^{\theta}}
\end{equation}
is attained. The supremum in \eqref{eq:C_1} is taken over all nonzero
admissible tempered distributions $\phi$ such that
$$
D^{\mu}\phi\in L^p(\mathbb{R}^d),\quad
D^{\mu+\sigma}\phi\in L^2(\mathbb{R}^d).
$$
\end{theorem}

\begin{proof}
We first note that $0<C_1<\infty$. Indeed, for $\vartheta\in(0,1)$, define
$$
\frac{1}{q_\vartheta}:=
(1-\vartheta)\left(\frac{1}{p}-\frac{\mu}{d}\right)
+\vartheta\left(\frac{1}{2}-\frac{\mu+\sigma}{d}\right).
$$
Since
$$
\frac{1}{q_\theta}=\frac{1}{q}\in(0,1),
$$
by continuity,
$$
0<\frac{1}{q_\vartheta}<1
$$
for all $\vartheta$ in a sufficiently small neighborhood of $\theta$.
For such $\vartheta$, the fractional Gagliardo--Nirenberg inequality
\cite{Chen2026} yields
\begin{equation}\label{eq:Sobolev+GN}
\|\phi\|_{L^{q_\vartheta}}
\lesssim
\|D^{\mu}\phi\|_{L^p}^{1-\vartheta}
\|D^{\mu+\sigma}\phi\|_{L^2}^{\vartheta}.
\end{equation}
In particular, taking $\vartheta=\theta$ gives $0<C_1<\infty$.

Recall that
$$
\eta=\sigma+\frac{d}{p}-\frac{d}{2}.
$$
We distinguish the cases $\eta\neq0$ and $\eta=0$.

\textbf{Case 1:} Assume that $\eta\neq0$.
Let $(\phi_n)$ be a maximizing sequence for \eqref{eq:C_1}.
We may first assume that
$$
\|D^{\mu}\phi_n\|_{L^p}=1.
$$
Let
$$
\bar{\phi}_n(x):=
\rho_n^{d/p-\mu}\phi_n(\rho_n x),
$$
where
$$
\rho_n:=\|D^{\mu+\sigma}\phi_n\|_{L^2}^{-1/\eta}
$$
Since the quotient in \eqref{eq:C_1} is invariant under this scaling,
after relabeling, we have
$$
\|D^{\mu}\phi_n\|_{L^p}
=
\|D^{\mu+\sigma}\phi_n\|_{L^2}
=1,
$$
and
$$
\|\phi_n\|_{L^q}\to C_1.
$$
As
$$
\frac{d}{d\vartheta}\frac{1}{q_\vartheta}=-\frac{\eta}{d}\neq0,
$$
the map $\vartheta\mapsto q_\vartheta$ is strictly monotone. Therefore, there exist
$$
0<\theta_-<\theta<\theta_+<1
$$
such that, after relabeling,
$$
1<p_1<q<p_2<\infty, \quad
\{p_1,p_2\}
=\{q_{\theta_-},q_{\theta_+}\}.
$$
Then
$$
\sup_n\left(\|\phi_n\|_{L^{p_1}}+\|\phi_n\|_{L^{p_2}}\right)
<\infty
$$
by \eqref{eq:Sobolev+GN}.
The $pqr$-lemma \cite{frohlich1986} provides constants $\delta,c>0$ such
that
$$
\left|
\{x\in\mathbb{R}^d:|\phi_n(x)|>\delta\}
\right|
\ge c.
$$
Since $(\phi_n)$ is bounded in
$\dot H^{\mu+\sigma}(\mathbb{R}^d)\cap L^q(\mathbb{R}^d)$,
the ‘compactness up to translations’ lemma \cite{Bellazzini2014} provides
$x_n\in\mathbb{R}^d$ and a nonzero function $\phi$ such that
\begin{equation}\label{eq:weak}
\widetilde{\phi}_n(x):=\phi_n(x+x_n)
\rightharpoonup\phi
\quad
\text{weakly in }
\dot H^{\mu+\sigma}(\mathbb{R}^d)\cap L^q(\mathbb{R}^d),
\end{equation}
up to a subsequence. 

Having established the weak limit in \eqref{eq:weak}, we next derive a.e. convergence of $\widetilde{\phi}_n$ and $D^{\mu}\widetilde{\phi}_n$. This allows us to apply the Br\'ezis--Lieb lemma to split the constraint $M$
and the nonlinear term $B$.  Together with the splitting of the quadratic term $A$,
these identities will be used to show that the weak limit attains the
sharp constant.

Now we return to a.e. convergence. By Lemma \ref{lem:local-compactness-Lp2},
 $$
\widetilde{\phi}_n\to\phi
\quad\text{a.e. in }\mathbb{R}^d.
 $$
Since $(D^\mu\widetilde{\phi}_n)$ is bounded in $L^p$ and
 $$
\|D^\sigma(D^\mu\widetilde{\phi}_n)\|_{L^2}=1,
 $$
the same lemma yields, after passing to a subsequence,
 $$
D^\mu\widetilde{\phi}_n\to f
\quad\text{a.e. in }\mathbb{R}^d.
 $$
Hence, by \cite[Proposition 4.7.12]{Bogachev2007},
 $$
D^\mu\widetilde{\phi}_n\rightharpoonup f
\quad\text{in }L^p(\mathbb{R}^d).
 $$
On the other hand,
$\widetilde{\phi}_n\rightharpoonup\phi$ in $L^q$ implies
 $$
D^\mu\widetilde{\phi}_n\to D^\mu\phi
\quad\text{in }\mathcal D'(\mathbb{R}^d).
 $$
Since weak convergence in $L^p$ also implies distributional
convergence, we have
$f=D^\mu\phi$ by uniqueness. 

By the Br\'ezis--Lieb lemma,
\begin{equation}\label{eq:BL-1}
\|\widetilde{\phi}_n\|_{L^q}^q
=
\|\phi\|_{L^q}^q
+
\|w_n\|_{L^q}^q
+o(1),
\end{equation}
and
\begin{equation}\label{eq:BL-2}
1=\|D^{\mu}\phi\|_{L^p}^p+\|D^{\mu}w_n\|_{L^p}^p+o(1),
\end{equation}
where
$$
w_n:=\widetilde{\phi}_n-\phi.
$$
Furthermore, weak convergence in
$\dot H^{\mu+\sigma}(\mathbb{R}^d)$ gives
\begin{equation}\label{eq:weak-conv}
1= \|D^{\mu+\sigma}\phi\|_{L^2}^2 + \|D^{\mu+\sigma}w_n\|_{L^2}^2 +o(1).
\end{equation}

Passing to the limit in
$$
\|\phi\|_{L^q}^q+\|w_n\|_{L^q}^q
\leq
C_1^q\left(
\|D^{\mu}\phi\|_{L^p}^{(1-\theta)q}
\|D^{\mu+\sigma}\phi\|_{L^2}^{\theta q}+\|D^{\mu}w_n\|_{L^p}^{(1-\theta)q}
\|D^{\mu+\sigma}w_n\|_{L^2}^{\theta q}
\right),
$$
and using \eqref{eq:BL-1}--\eqref{eq:weak-conv}, we obtain
$$
1
\leq
b^\alpha a^\beta
+
(1-b)^\alpha(1-a)^\beta,
$$
where
$$
b:=\|D^{\mu}\phi\|_{L^p}^p, \quad a:=\|D^{\mu+\sigma}\phi\|_{L^2}^2,
$$
and
$$
\alpha:=\frac{(1-\theta)q}{p}, \quad \beta:=\frac{\theta q}{2}.
$$
On the other hand,
$$
\gamma:=\alpha+\beta >1
$$
by the scaling relation.
Since $\alpha,\beta>0$, 
$$
\begin{aligned}
b^\alpha a^\beta+
(1-b)^\alpha(1-a)^\beta &\leq
\left(
b^{\alpha/\gamma}a^{\beta/\gamma}
+
(1-b)^{\alpha/\gamma}(1-a)^{\beta/\gamma}
\right)^\gamma\\
&\leq
\left(
(b+1-b)^{\alpha/\gamma}
(a+1-a)^{\beta/\gamma}
\right)^\gamma
=1
\end{aligned}
$$
by H\"older's inequality.
Together with the reverse inequality obtained above, we have 
\begin{equation}\label{eq:equality}
1=
b^\alpha a^\beta
+
(1-b)^\alpha(1-a)^\beta.
\end{equation}
As $a,b>0$, equality in
$$
\left(b^{\alpha/\gamma}a^{\beta/\gamma}\right)^\gamma
+
\left((1-b)^{\alpha/\gamma}(1-a)^{\beta/\gamma}\right)^\gamma
\leq
\left(
b^{\alpha/\gamma}a^{\beta/\gamma}
+
(1-b)^{\alpha/\gamma}(1-a)^{\beta/\gamma}
\right)^\gamma
$$
and
$$
b^{\alpha/\gamma}a^{\beta/\gamma}>0
$$
imply
$$
(1-b)^{\alpha/\gamma}(1-a)^{\beta/\gamma}=0.
$$
Then from \eqref{eq:equality}, we obtain
$$
b^\alpha a^\beta=1,
$$
which in turn forces
$$
a=b=1 .
$$
Consequently, by the preceding Brezis--Lieb decompositions
\eqref{eq:BL-2}--\eqref{eq:weak-conv}, we derive
$$
\|D^{\mu} w_n\|_{L^p}, \,
\|D^{\mu+\sigma}w_n\|_{L^2}\to 0.
$$
The sharp inequality then gives $\|w_n\|_{L^q}\to 0$, and
\eqref{eq:BL-1} yields
$$
\|\phi\|_{L^q}=C_1.
$$
Therefore, $\phi$ attains $C_{*}$ and is an optimizer.

\textbf{Case 2:} Assume that $\eta=0$.
Then
$$
\sigma=\frac{d}{2}-\frac{d}{p}>0
$$
implies $p>2$. The scaling relation reduces to
\begin{equation}\label{eq:Sobolev-scaling}
\frac{1}{q}
=
\frac{1}{p}-\frac{\mu}{d}
=
\frac{1}{2}-\frac{\mu+\sigma}{d}.
\end{equation}
Setting
$$
\tau:=\mu+\sigma,
$$
we have
$$
0<\tau<\frac{d}{2},
\quad
q=\frac{2d}{d-2\tau}.
$$

Let $(\phi_n)$ be a maximizing sequence for \eqref{eq:C_1}. We normalize
$$
\|D^{\tau}\phi_n\|_{L^2}=1.
$$
Applying the Sobolev inequality twice with \eqref{eq:Sobolev-scaling},
$$
\|\phi_n\|_{L^q}
\lesssim
\|D^{\mu}\phi_n\|_{L^p}
\lesssim
\|D^{\tau}\phi_n\|_{L^2}
=1.
$$
Since
\begin{equation}\label{eq:max-sequence}
\frac{\|\phi_n\|_{L^q}}
{\|D^{\mu}\phi_n\|_{L^p}^{1-\theta}}
\to C_1,
\end{equation}
we have
$$
\|D^{\mu}\phi_n\|_{L^p}\gtrsim1.
$$
Therefore, after passing to a subsequence,
\begin{equation}\label{eq:D-mu-conv}
\|D^{\mu}\phi_n\|_{L^p}^p\to m
\end{equation}
for some $m>0$. Combining \eqref{eq:max-sequence}--\eqref{eq:D-mu-conv}, we have
$$
\|\phi_n\|_{L^q}
\rightarrow
C_1m^{\frac{1-\theta}{p}}.
$$
After discarding finitely many terms,
$$
\inf_n\|\phi_n\|_{L^q}>0.
$$

By \cite[Lemma 2]{Palatucci2014}, there exist
$x_n\in\mathbb{R}^d$, $\lambda_n>0$, and a nonzero $\phi$ such that
$$
\widetilde{\phi}_n
\rightharpoonup\phi
\quad\text{weakly in }\dot H^\tau(\mathbb{R}^d),
$$
where
\begin{equation}\label{eq:Palatucci}
\widetilde{\phi}_n(x)
:=
\lambda_n^{d/2-\tau}
\phi_n(x_n+\lambda_n x).
\end{equation}
Since
$$
\frac{d}{2}-\tau
=
\frac{d}{p}-\mu
=
\frac{d}{q},
$$
the transformation \eqref{eq:Palatucci} preserves all three norms:
$$
\|\widetilde{\phi}_n\|_{L^q}
=
\|\phi_n\|_{L^q},
\quad
\|D^{\mu}\widetilde{\phi}_n\|_{L^p}
=
\|D^{\mu}\phi_n\|_{L^p},
\quad
\|D^{\tau}\widetilde{\phi}_n\|_{L^2}
=
\|D^{\tau}\phi_n\|_{L^2}.
$$
Thus, we can replace $\phi_n$ by $\widetilde{\phi}_n$.

As above, Lemma \ref{lem:local-compactness-Lp2} and the
Br\'ezis--Lieb lemma yield, with $w_n:=\phi_n-\phi$,
$$
\|\phi_n\|_{L^q}^q
=
\|\phi\|_{L^q}^q+\|w_n\|_{L^q}^q+o(1),
$$
$$
m
=
\|D^{\mu}\phi\|_{L^p}^p
+
\|D^{\mu}w_n\|_{L^p}^p
+o(1),
$$
and
$$
1
=
\|D^{\tau}\phi\|_{L^2}^2
+
\|D^{\tau} w_n\|_{L^2}^2
+o(1).
$$
The rest of the proof is similar to Case 1.
\end{proof}

Having established the necessary preliminary results, we are now ready to prove
Theorem \ref{thm:GN}.

\begin{proof}[Proof of Theorem \ref{thm:GN}]
If $s>s_1$, setting
$$
v:=D^{s_1}\phi,\quad
\mu:=s-s_1>0,\quad
\sigma:=s_2-s_1>0,
$$
we see that the quotient in \eqref{eq:best-const} becomes
$$
\frac{\|D^{\mu} v\|_{L^q}}
{\|v\|_{L^p}^{1-\theta}\|D^{\sigma} v\|_{L^2}^{\theta}}.
$$
Moreover,
$$
\frac{\mu}{\sigma}<\theta
$$
is equivalent to \eqref{eq:strict-interpolation}.
Therefore, the result follows from \eqref{eq:GN-BFV}. 

If $s\leq s_1$, we set
$$
v:=D^{s}\phi,\quad
\mu:=s_1-s\geq 0,\quad
\sigma:=s_2-s_1>0.
$$
Then the quotient in \eqref{eq:best-const} becomes
$$
\frac{\|v\|_{L^q}}
{\|D^{\mu} v\|_{L^p}^{1-\theta}
 \|D^{\mu+\sigma}v\|_{L^2}^{\theta}},
$$
which is covered by Theorem \ref{thm:GN-optimizer}. In both cases,
invertibility of the corresponding power of $D$ on
$\mathcal{S}'(\mathbb{R}^d)/\mathcal{P}(\mathbb{R}^d)$ transfers the optimizer
back to the original quotient.
\end{proof}

\section{Proof of Theorem \ref{thm:subcritical}}

\begin{proof}[Proof of Theorem \ref{thm:subcritical}]
Let $\psi \in \mathcal{S}_{\lambda}$. Recall that
$$
\alpha=\frac{\theta q}{2},
\quad
\beta=\frac{(1-\theta)q}{p}.
$$
We rewrite sharp Gagliardo--Nirenberg inequality \eqref{eq:GN-sharp} as follows
\begin{equation}\label{eq:GN-in-A,B}
B(\psi)
\leq 
C_{*}^q
\lambda^\beta
A(\psi)^\alpha.
\end{equation}
Then
\begin{equation}\label{eq:E-bound-by-A}
E(\psi)
\geq 
\frac{1}{2} A(\psi)
-
\frac{C_{*}^q}{q}\lambda^\beta A(\psi)^\alpha.
\end{equation}
We see that the lower bound for $E(\psi)$ depends only on $A(\psi)$, so we can minimize the auxiliary function
$$
f_{\lambda}(t)
:=
\frac{1}{2}t
-
\frac{C_{*}^q}{q}\lambda^\beta t^\alpha,
\quad t\geq  0.
$$
The limits
$$
\lim_{t\to 0^+}f_{\lambda}'(t)=-\infty,
\quad
\lim_{t\to+\infty}f_{\lambda}'(t)=\frac{1}{2}
$$
imply that $f_{\lambda}'$ has at least one zero in $(0,\infty)$.
Since 
$f_{\lambda}'$ is strictly increasing,  there is only one zero.
One can show that
$$
t_{\lambda}
:=
\left(
\theta C_{*}^q
\lambda^\beta
\right)^{\frac{1}{1-\alpha}}
$$
is the unique minimizer of $f_{\lambda}$. Thus,
\begin{equation}\label{eq:lower-bound-subcritical-sharp}
I_{\lambda}
\geq  
-\frac{2-q\theta}{2q\theta}
\left(
\theta C_{*}^q
\lambda^{\frac{(1-\theta)q}{p}}
\right)^{\frac{2}{2-q\theta}}.
\end{equation}

It remains to show that equality is attained in \eqref{eq:lower-bound-subcritical-sharp}. Let $\Phi\neq0$ be an optimizer
for \eqref{eq:GN-in-A,B}. Then
$$
\widetilde{\psi}
:=
\left(
\frac{\lambda}{M(\Phi)}
\right)^{\frac{1}{p}}\Phi
$$
belongs to $\mathcal{S}_{\lambda}$, and  is still an optimizer for \eqref{eq:GN-in-A,B}. 
Setting
$$
\rho_{\lambda}
:=
\left(
\frac{t_{\lambda}}{A(\widetilde{\psi})}
\right)^{\frac{1}{2\eta}},\quad \widetilde\psi_{\lambda}:=\rho_{\lambda}\star\widetilde{\psi}
$$
we have
\begin{equation}\label{eq:A-t-lambda}
\widetilde\psi_{\lambda}\in\mathcal{S}_{\lambda}, \quad A(\widetilde\psi_{\lambda})
=
t_{\lambda}.
\end{equation}
Since $\widetilde\psi_{\lambda}$ is still an
optimizer for \eqref{eq:GN-in-A,B}, we have 
\begin{equation}\label{eq:B-t-lambda}
B(\psi_{\lambda})
=
\frac{1}{\theta}t_{\lambda}.
\end{equation}
Combining \eqref{eq:A-t-lambda}--\eqref{eq:B-t-lambda}, we have
$$
\begin{aligned}
E(\psi_{\lambda})=
-\frac{2-q\theta}{2q\theta}t_{\lambda}=I_{\lambda}.
\end{aligned}
$$
Therefore, $I_{\lambda}$ is attained, and \eqref{eq:Ilambda-explicit-subcritical}
follows.
\end{proof}

\section{Proof of Theorem \ref{thm:critical}}
Let $\psi \in \mathcal{S}_{\lambda}$. In the $M$-critical case, the inequality \eqref{eq:GN-sharp} becomes
\begin{equation}\label{eq:GN-critical}
B(\psi)
\leq 
C_{*}^q
\lambda^{\frac{(1-\theta)q}{p}}
A(\psi).
\end{equation}
Consequently,
\begin{equation}\label{eq:E-bounded-by-A}
E(\psi)
\geq  
\left(
\frac{1}{2}
-
\frac{C_{*}^q}{q}
\lambda^{\frac{(1-\theta)q}{p}}
\right)A(\psi).
\end{equation}

\begin{proof}[Proof of Theorem \ref{thm:critical}]

Let $0<\lambda<\lambda_*$. 
By \eqref{eq:E-bounded-by-A}, there exists $c_{\lambda}>0$
such that
\begin{equation}\label{eq:coerc-crit}
E(\psi)\geq   c_{\lambda} A(\psi).
\end{equation}
In particular,
$$
I_{\lambda}\geq  0.
$$
On the other hand, 
$$
I_{\lambda}\leq  E(\rho\star\psi)
=
\rho^{2\eta}E(\psi)\to 0
\quad\text{as }\rho\to 0^+.
$$
Thus,
$$
I_{\lambda}=0.
$$
We now show that $I_{\lambda}$ is not attained.
Suppose, by contradiction, that there exists
$\psi_{\lambda}\in\mathcal{S}_{\lambda}$ such that
$$
E(\psi_{\lambda})=I_{\lambda}=0.
$$
But \eqref{eq:coerc-crit} implies
$$
A(\psi_{\lambda})=0.
$$
Hence,
$$
D^{s_2}\psi_{\lambda}=0.
$$
Since $s_2>0$, the Fourier transform of $\psi_{\lambda}$ is supported at
the origin. Thus, $\psi_{\lambda}$ is a polynomial. The condition
$D^{s_1}\psi_{\lambda}\in L^p(\mathbb{R}^d)$ then forces
$$
M(\psi_{\lambda})=0,
$$
which is impossible. Therefore, $I_{\lambda}$ is not
attained for $0<\lambda<\lambda_*$.

 Next, let $\lambda=\lambda_*$. By \eqref{eq:E-bounded-by-A},
$$
I_{\lambda_*}\geq  0.
$$
Let $\Phi\neq0$ be an optimizer for \eqref{eq:GN-critical}.
Setting
$$
\psi_*
:=
\left(
\frac{\lambda_*}{M(\Phi)}
\right)^{\frac{1}{p}}\Phi,
$$
we see that $\psi_*\in\mathcal{S}_{\lambda_*}$, and $\psi_*$ is still an optimizer
for \eqref{eq:GN-sharp}. Hence,
$$
B(\psi_*)=
C_{*}^q
\lambda_*^{\frac{(1-\theta)q}{p}}
A(\psi_*).
$$
Since
$$
C_{*}^q
\lambda_*^{\frac{(1-\theta)q}{p}}
=
\frac{q}{2},
$$
it follows that
$$
I_{\lambda_*}=E(\psi_*)
=
0.
$$

Finally, let $\lambda>\lambda_*$ and
$$
\widetilde{\psi}
:=
\left(
\frac{\lambda}{M(\Phi)}
\right)^{\frac{1}{p}}\Phi.
$$
Then $\widetilde{\psi}\in\mathcal{S}_{\lambda}$ and
$$
B(\widetilde{\psi})
=
C_{*}^q
\lambda^{\frac{(1-\theta)q}{p}}
A(\widetilde{\psi}).
$$
Since 
$$
E(\widetilde{\psi})
=
\left(
\frac{1}{2}
-
\frac{C_{*}^q}{q}
\lambda^{\frac{(1-\theta)q}{p}}
\right)A(\widetilde{\psi})
<0,
$$
it follows that
$$
E(\rho\star\widetilde{\psi})=
\rho^{2\eta}E(\widetilde{\psi})\to-\infty
\quad\text{as }\rho\to+\infty.
$$
Therefore,
$$
I_{\lambda}=-\infty.
$$
\end{proof}

\section{Proof of Theorem \ref{thm:supercritical}}\label{sec6}

In the $M$-supercritical case, for every $\psi\in\mathcal{S}_{\lambda}$,
$$
E(\rho\star\psi)\to -\infty
\quad \text{as } \rho\to\infty.
$$
Consequently,
\begin{equation*}\label{eq:supercr.}
I_{\lambda}:=\inf\{E(\psi):\psi\in\mathcal{S}_{\lambda}\}=-\infty.
\end{equation*}
To overcome this, we use the scaling \eqref{eq:scaling}
to identify a natural constraint, namely, the Pohozaev identity.

Given $\omega\in\mathbb{R}$, recall
$$
S_{\omega}(\psi)=
E(\psi)+\frac{\omega}{p}M(\psi).
$$
A weak solution of
\begin{equation}\label{eq:weak-form}
D^{2s_2}\psi - D^{s}\left(|D^{s}\psi|^{q-2}D^{s}\psi\right)
+\omega D^{s_1}\left(|D^{s_1}\psi|^{p-2}D^{s_1}\psi\right)
=0
\end{equation}
is a critical point of $S_{\omega}$ in $\mathcal{X}$.

Let $\psi$ be a critical point of $S_{\omega}$. Since the scaling \eqref{eq:scaling} preserves $M$,
the scalar fiber map
$$
\rho\mapsto S_{\omega}(\rho\star\psi)
$$
has a critical point at $\rho=1$. Differentiating at $\rho=1$, we obtain
\begin{equation}\label{eq:Pohozaev-identity}
\begin{aligned}
0
&=
\left.\frac{d}{d\rho}S_{\omega}(\rho\star\psi)\right|_{\rho=1}  \\
&=
\left.
\frac{d}{d\rho}
\left(
\frac{1}{2}\rho^{2\eta}A(\psi)
-
\frac{1}{q}\rho^{q\theta\eta}B(\psi)
\right)
\right|_{\rho=1} \\
&=
\eta\left(A(\psi)-\theta B(\psi)\right).
\end{aligned}
\end{equation}
This motivates the definition of the Pohozaev functional
$$
P(\psi)=A(\psi)-\theta B(\psi).
$$

Now we show that $\omega>0$ is necessary for the existence of a nontrivial weak
solution of \eqref{eq:weak-form} satisfying the Pohozaev identity. Taking $\psi$ as a test function
in the weak formulation, we get
\begin{equation*}
A(\psi)-B(\psi)+\omega M(\psi)=0.
\end{equation*}
By the Pohozaev identity \eqref{eq:Pohozaev-identity},
$$
\omega M(\psi)=\frac{1-\theta}{\theta}A(\psi)>  0.
$$
If $\omega\leq 0$, then $A(\psi)=0$, and hence $\psi=0$ in $\mathcal{X}$. 
Therefore, any nontrivial weak solution of \eqref{eq:weak-form} satisfying the Pohozaev identity must have
$$
\omega>0.
$$
The next lemma shows that every scaling fiber intersects the Pohozaev identity
in exactly one point, and that this point is the unique maximizer of the energy functional
along the fiber.

\begin{lemma}\label{lem:Pohozaev projection}
Assume that $q\theta>2$. Then for every $\psi\in\mathcal{S}_{\lambda}$, there exists a unique
\begin{equation}\label{eq:rho}
\rho(\psi):=\left(\frac{A(\psi)}{\theta B(\psi)}\right)^{\frac{1}{\eta(q\theta-2)}}
\end{equation}
such that
$$
P(\rho(\psi)\star\psi)=0.
$$
Moreover,
\begin{equation*}\label{eq:max-rho}
E\left(\rho(\psi)\star\psi\right)=\max_{\rho>0}E(\rho\star\psi)=\frac{q\theta-2}{2q\theta}
\theta^{-\frac{2}{q\theta-2}}
\frac{
A(\psi)^{\frac{q\theta}{q\theta-2}}
}{
B(\psi)^{\frac{2}{q\theta-2}}
}.
\end{equation*}
\end{lemma}

\begin{proof}
Similar to the arguments in the proof of Theorem \ref{thm:critical}, one can show that $B(\psi)>0$. 
For given $\psi\in\mathcal{S}_{\lambda}$ and $\rho>0$, we introduce the following function
$$
g(\rho):=E(\rho\star\psi)=
\frac{1}{2}\rho^{2\eta}A(\psi)-
\frac{1}{q}\rho^{q\theta\eta}B(\psi).
$$
Then
$$
g'(\rho)=0\quad\Longleftrightarrow\quad P(\rho\star\psi)=0.
$$
Since $q\theta>2$, this equation has the unique positive solution given by \eqref{eq:rho}.
Moreover,
$$
\begin{cases}
g'(\rho)>0, \quad& \text{if } 0<\rho<\rho(\psi),\\
g'(\rho)<0, \quad & \text{if } \rho>\rho(\psi).
\end{cases}
$$
Thus, $\rho(\psi)$ is the unique global maximizer of $g$.

It remains to compute the maximum of $g$. Since
$$
\begin{aligned}
E(\rho(\psi)\star\psi)=
\frac{q\theta-2}{2q\theta}
\rho(\psi)^{2\eta}A(\psi),
\end{aligned}
$$
it follows that
$$
E(\rho(\psi)\star\psi)=\frac{q\theta-2}{2q\theta}
\theta^{-\frac{2}{q\theta-2}}
\frac{
A(\psi)^{\frac{q\theta}{q\theta-2}}
}{
B(\psi)^{\frac{2}{q\theta-2}}
}.
$$
This completes the proof.
\end{proof}

The previous lemma motivates the following definition. Recall the Pohozaev manifold
$$
\mathcal{P}_{\lambda}=\left\{\psi \in \mathcal{S}_{\lambda} \mid P(\psi)=0\right\}.
$$
By Lemma \ref{lem:Pohozaev projection}, the set $\mathcal{P}_{\lambda}$ is nonempty, and  
$$
\{\rho\star\psi:\rho>0\}
$$
intersects $\mathcal{P}_{\lambda}$ at exactly one point 
for every $\psi\in\mathcal{S}_{\lambda}$.

Let $\psi\in\mathcal{P}_{\lambda}$. Then
\begin{equation*}\label{eq:energy-on-Pohozaev}
E(\psi)
=
\frac{q\theta-2}{2q\theta}A(\psi).
\end{equation*}
Since
$$
A(\psi)
\leq 
\theta C_{*}^q
A(\psi)^{\frac{\theta q}{2}}
\lambda^{\frac{(1-\theta)q}{p}},
$$ 
we have
\begin{equation}\label{eq:lower bound of E}
E(\psi)
\geq  
\frac{q\theta-2}{2q\theta}
\left(
\frac{1}
{
\theta C_{*}^q
\lambda^{\frac{(1-\theta)q}{p}}
}
\right)^{
\frac{1}{\frac{\theta q}{2}-1}
}
>0.
\end{equation}
Therefore, it is natural to consider
\begin{equation}\label{eq-min-on-Pohozaev}
m_{\lambda}
:=
\inf_{\psi\in\mathcal{P}_{\lambda}}E(\psi).
\end{equation}
Now let $\Phi\neq0$ be an optimizer for \eqref{eq:GN-sharp}.
Then
$$
\widetilde{\psi}
:=
\left(
\frac{\lambda}{M(\Phi)}
\right)^{\frac{1}{p}}\Phi
$$
belongs to $\mathcal{S}_{\lambda}$, and is still an optimizer for \eqref{eq:GN-sharp}. By Lemma
\ref{lem:Pohozaev projection}, there exists a unique number
$\rho(\widetilde{\psi})>0$ such that
$$
P(\rho(\widetilde{\psi})\star\widetilde{\psi})=0.
$$
Since 
$
\psi_{\lambda}:=\rho(\widetilde{\psi})\star\widetilde{\psi}\in \mathcal{P}_{\lambda}
$
 is also an optimizer for \eqref{eq:GN-sharp}, we have
\begin{equation}\label{eq:E-at-minimizer}
E(\psi_{\lambda})
=
\frac{q\theta-2}{2q\theta}
\left(
\frac{1}
{
\theta C_{*}^q
\lambda^{\frac{(1-\theta)q}{p}}
}
\right)^{
\frac{1}{\frac{\theta q}{2}-1}
}=m_{\lambda}
\end{equation}
by \eqref{eq:lower bound of E}.
The Pohozaev manifold $\mathcal{P}_{\lambda}$ was introduced as an auxiliary
constraint in order to overcome the fact that $E$ is unbounded from below on
$\mathcal{S}_{\lambda}$. We now show that this auxiliary constraint is natural:
every minimizer of $E$ on $\mathcal{P}_{\lambda}$ is automatically a 
critical point of $E$ on $\mathcal{S}_{\lambda}$. 

\begin{proof}[Proof of Theorem \ref{thm:supercritical}]
Based on Lemma \ref{lem:Pohozaev projection},  we define a $C^1$ functional on $\mathcal{S}_{\lambda}$ as follows
$$
J(\psi):=
\frac{q\theta-2}{2q\theta}
\theta^{-\frac{2}{q\theta-2}}
A(\psi)^{\frac{q\theta}{q\theta-2}}
B(\psi)^{-\frac{2}{q\theta-2}}.
$$
In other words,
$$
J(\psi)=E(\rho(\psi)\star\psi)
$$
by Lemma \ref{lem:Pohozaev projection}. A direct computation yields
\begin{equation}\label{eq:J'=E'}
J'(\psi)=
E'(\psi)\quad\text{for all }\psi\in\mathcal{P}_{\lambda}.
\end{equation}

Let $\psi_{\lambda}\in\mathcal{P}_{\lambda}$ be a minimizer for \eqref{eq-min-on-Pohozaev}. Since $\rho(\psi)\star\psi\in
\mathcal{P}_{\lambda}$ for every $\psi\in\mathcal{S}_{\lambda}$,
it follows that
\begin{equation}\label{eq:J-E}
J(\psi)
=
E(\rho(\psi)\star\psi)
\ge
\inf_{\varphi\in\mathcal{P}_{\lambda}}E(\varphi)
=
E(\psi_{\lambda}).
\end{equation}
Since $\psi_{\lambda}\in\mathcal{P}_{\lambda}$, we have
$\rho(\psi_{\lambda})=1$. Therefore, combining 
$$
J(\psi_{\lambda})=E(\psi_{\lambda})
$$
with \eqref{eq:J-E}, we see that $\psi_{\lambda}$ is a minimizer of $J$ on
$\mathcal{S}_{\lambda}$. Consequently, in view of \eqref{eq:J'=E'}, we have
$$
\langle J'(\psi_{\lambda}),h\rangle=\langle E'(\psi_{\lambda}),h\rangle=0
\quad
\text{for every } h\in T_{\psi_{\lambda}}\mathcal{S}_{\lambda}.
$$
Thus, $\psi_{\lambda}$ is a critical point of $E$ on $\mathcal{S}_{\lambda}$.

Lastly, from \eqref{eq:E-at-minimizer}, it is immediate that
$$
E(\psi_{\lambda})\to+\infty
\quad\text{as }\lambda\to 0^+,
$$
and
$$
E(\psi_{\lambda})\to 0
\quad\text{as }\lambda\to+\infty.
$$
\end{proof}

\section{Proof of Theorem \ref{thm:eta-zero}}

\begin{proof}[Proof of Theorem \ref{thm:eta-zero}]
The scaling relation \eqref{eq:scaling-relation} reduces to
\begin{equation}\label{eq:eta-zero-scaling}
\frac{1}{q}
=
\frac{1}{p}+\frac{s-s_1}{d}.   
\end{equation}
Since $s\leq s_1$, it follows that
$$
q\geq p>2.
$$
If $s<s_1$, then
$$
0<s_1-s<\frac{d}{p}.
$$
Using the Hardy--Littlewood--Sobolev inequality (see e.g. \cite[Theorem 1.2.3]{Grafakos2014modern}), we get
$$
\|D^s\psi\|_{L^q}
=
\|D^{-(s_1-s)}D^{s_1}\psi\|_{L^q}
\lesssim
\|D^{s_1}\psi\|_{L^p}.
$$
If $s=s_1$, then $q=p$ and
$$
\|D^s\psi\|_{L^q}
=
\|D^{s_1}\psi\|_{L^p}.
$$
In either case,
\begin{equation}\label{eq:eta-zero-endpoint}
\|D^s\psi\|_{L^q}
\lesssim
\|D^{s_1}\psi\|_{L^p}.
\end{equation}
From \eqref{eq:eta-zero-endpoint}, we see that
\begin{equation}\label{eq:eta-zero-coercive}
E(\psi)
\geq
\frac{1}{2}A(\psi)-C\lambda^{q/p}.
\end{equation}
Thus, $I_{\lambda}>-\infty$ and every minimizing sequence is bounded in
$\mathcal{X}$.

We next prove the strict subadditivity. Let $\tau>1$ and
$\psi\in\mathcal{S}_{\lambda}$. Then
$\tau^{1/p}\psi\in\mathcal{S}_{\tau\lambda}$
and
\begin{equation}\label{eq:strict-scaling}
\tau E(\psi)-E(\tau^{1/p}\psi)
=
\frac{\tau-\tau^{2/p}}{2}A(\psi)
+
\frac{\tau^{q/p}-\tau}{q}B(\psi)\geq 0.
\end{equation}
Since
\begin{equation*}\label{eq:critical Sobolev}
\frac{1}{p}
=
\frac{1}{2}-\frac{s_2-s_1}{d}, 
\end{equation*}
the Sobolev inequality yields
\begin{equation*}\label{eq:A-lower-eta-zero}
A(\psi)\geq c\,\lambda^{2/p}.
\end{equation*}
It follows from \eqref{eq:strict-scaling} that
\begin{equation}\label{eq:some-bound}
E(\tau^{1/p}\psi)\leq \tau E(\psi)-\frac{\tau-\tau^{2/p}}{2}
c\,\lambda^{2/p}<\tau E(\psi)
\end{equation}
Applying \eqref{eq:some-bound} to a minimizing sequence for $I_{\lambda}$ gives
\begin{equation*}\label{eq:I-strict-scaling}
I_{\tau\lambda}<\tau I_{\lambda}.
\end{equation*}
Since
$$
\lambda\mapsto\frac{I_{\lambda}}{\lambda}
$$
is strictly decreasing, we have 
\begin{equation}\label{eq:strict-subadditivity}
I_{\lambda}<I_{\mu}+I_{\lambda-\mu}\quad \text{for every }0<\mu<\lambda.
\end{equation}

Let $(\psi_n)\subset\mathcal{S}_{\lambda}$ be a minimizing sequence for $I_{\lambda}$. Since
$$
v_n:=D^{s_1}\psi_n
$$
is bounded in $\dot{H}^{s_2-s_1}(\mathbb{R}^d)$ and
$$
\|v_n\|_{L^p}^p=\lambda,
$$
it follows from \cite[Lemma 2]{Palatucci2014} that there exist
$x_n\in\mathbb{R}^d$ and $\rho_n>0$ such that
\begin{equation}\label{eq:Palatucci-1}
\widetilde v_n(x)
:=
\rho_n^{d/p}v_n(x_n+\rho_n x)
\rightharpoonup v\neq0 \quad \text{ in }\dot H^{s_2-s_1}(\mathbb{R}^d),
\end{equation}
up to a sequence. Accordingly, we define
$$
\widetilde\psi_n(x)
:=
\rho_n^{d/p-s_1}
\psi_n(x_n+\rho_n x),\quad \psi(x):=D^{-s_1}v(x).
$$
Then
$$
D^{s_1}\widetilde\psi_n=\widetilde v_n.
$$
Using $\eta=0$ and \eqref{eq:eta-zero-scaling}, we see that
$$
M(\widetilde\psi_n)=\lambda,
\quad
A(\widetilde\psi_n)=A(\psi_n),
\quad
B(\widetilde\psi_n)=B(\psi_n).
$$
So $(\widetilde\psi_n)$ is also a minimizing sequence for $I_\lambda$, but we relabel it
as $(\psi_n)$ for convenience. Thus, \eqref{eq:Palatucci-1} can be written as
$$
D^{s_1}\psi_n\rightharpoonup D^{s_1}\psi
\quad\text{in }\dot H^{s_2-s_1}(\mathbb{R}^d).
$$
Moreover, the Sobolev embedding
$$
\dot H^{s_2-s_1}(\mathbb{R}^d)
\hookrightarrow L^p(\mathbb{R}^d)
$$
implies
$$
D^{s_1}\psi_n\rightharpoonup D^{s_1}\psi
\quad\text{in }L^p(\mathbb{R}^d).
$$

To apply the Br\'ezis--Lieb lemma to the functionals $M$ and $B$, we next establish
a.e. convergence of $D^{s_1}\psi_n$ and $D^s\psi_n$.
By Lemma \ref{lem:local-compactness-Lp2}, after passing to a further
subsequence,
$$
D^{s_1}\psi_n\to f
\quad\text{a.e. in }\mathbb{R}^d
$$
for some $f$. Since $(D^{s_1}\psi_n)$ is bounded in $L^p$,
$$
D^{s_1}\psi_n\rightharpoonup f
\quad\text{in }L^p(\mathbb{R}^d)
$$
by \cite[Proposition~4.7.12]{Bogachev2007}.
The uniqueness of the weak limit gives
$$
f=D^{s_1}\psi.
$$

Next, by the Hardy--Littlewood--Sobolev inequality when $s<s_1$,
and trivially when $s=s_1$, the operator
$$
D^{s-s_1}:L^p(\mathbb{R}^d)\to L^q(\mathbb{R}^d)
$$
is bounded. Therefore,
$$
D^s\psi_n\rightharpoonup D^s\psi
\quad\text{in }L^q(\mathbb{R}^d).
$$
Since
$$
\sup_n
\|D^{s_2-s}(D^s\psi_n)\|_{L^2}<\infty,
$$
it follows from  Lemma \ref{lem:local-compactness-Lp2} that there exists some $g$ such that
$$
D^s\psi_n\to g
\quad\text{a.e. in }\mathbb{R}^d,
$$
up to a subsequence. Since $(D^s\psi_n)$ is bounded in $L^q(\mathbb{R}^d)$,
\cite[Proposition 4.7.12]{Bogachev2007} yields
$$
D^s\psi_n\rightharpoonup g
\quad\text{in }L^q(\mathbb{R}^d).
$$
By the uniqueness of the weak limit,
$$
g=D^s\psi.
$$

Having established the necessary a.e. convergence, we are ready to apply the Br\'ezis--Lieb lemma. Let
$$
w_n:=\psi_n-\psi.
$$
The Br\'ezis--Lieb lemma gives
\begin{equation}\label{eq:BL-eta-zero}
\lambda
=
M(\psi)
+
M(w_n)
+o(1),
\end{equation}
and
\begin{equation}\label{eq:BL-B-eta-zero}
B(\psi_n)
=
B(\psi)+B(w_n)+o(1).
\end{equation}
Moreover, by weak convergence in $\dot{H}^{s_2}(\mathbb{R}^d)$,
\begin{equation}\label{eq:A-splitting-eta-zero}
A(\psi_n)
=
A(\psi)+A(w_n)+o(1).
\end{equation}
Combining \eqref{eq:BL-B-eta-zero}--\eqref{eq:A-splitting-eta-zero},
\begin{equation}\label{eq:E-splitting-eta-zero}
E(\psi_n)
=
E(\psi)+E(w_n)+o(1).
\end{equation}

We now rule out a nontrivial splitting of the constraint.
It is clear that
$$
0<\mu:=M(\psi)\leq\lambda.
$$
Suppose, on the contrary,  $\mu\neq \lambda$. By
\eqref{eq:BL-eta-zero},
$$
M(w_n)\to\lambda-\mu.
$$
It is clear that
$$
c_n
:=
\left(
\frac{\lambda-\mu}
{M(w_n)}
\right)^{1/p} \to 1\quad \text{and}\quad c_n w_n\in\mathcal{S}_{\lambda-\mu}.
$$
Thus,
$$
E(c_nw_n)-E(w_n)
=
\frac{c_n^2-1}{2}A(w_n)
-
\frac{c_n^q-1}{q}B(w_n)
\to 0.
$$
Since
$$
\liminf_{n\to\infty}E(w_n)=\liminf_{n\to\infty}E(c_nw_n)
\geq I_{\lambda-\mu},
$$
the splitting identity \eqref{eq:E-splitting-eta-zero} yields
$$
I_{\lambda}\geq I_{\mu}+I_{\lambda-\mu},
$$
contradicting \eqref{eq:strict-subadditivity}. Therefore, 
\begin{equation}\label{eq:no-splitting}
M(\psi)=\lambda.
\end{equation}

It remains to prove that $\psi$ is a minimizer for $I_{\lambda}$.
Combining \eqref{eq:BL-eta-zero} and \eqref{eq:no-splitting}, we derive
$$
\|D^{s_1}w_n\|_{L^p}\to 0.
$$
Then, by \eqref{eq:eta-zero-endpoint},
$$
B(w_n)=o(1).
$$
Consequently, \eqref{eq:E-splitting-eta-zero} becomes
$$
I_{\lambda}
=
E(\psi)+\frac{1}{2}A(w_n)+o(1).
$$
Since 
$$
E(\psi)\geq I_{\lambda},
$$
we have
$$
A(w_n)\to 0,
\quad
E(\psi)=I_{\lambda}.
$$
Thus, $\psi$ attains $I_{\lambda}$, and $\psi_n\to\psi$ strongly in
$\mathcal{X}$.
\end{proof}

\section{Acknowledgments}
This research was funded by the Nazarbayev University Faculty Development Competitive Research Grants Program 040225FD4702. 

\section{Conflict of interest}
The authors declare that they have no conflict of interest.

\section{Data availability}
The manuscript has no associated data.

\addcontentsline{toc}{chapter}{Bibliography}
%uncomment next line to change bibliography name to references
%\renewcommand{\bibname}{References}
\bibliography{refs}      %use a bibtex bibliography file refs.bib

\bibliographystyle{abbrv}  %use the plain bibliography style unsrt apalike siam

\end{document}